\documentclass[11pt, a4paper, reqno]{amsart}
\usepackage{latexsym}
\usepackage{amsthm, amssymb, amsfonts}

\usepackage{algorithm,algorithmicx}
\floatname{algorithm}{Procedure}
\usepackage{algpseudocode}
\algblock{Input}{EndInput}
\algnotext{EndInput}
\algblock{Output}{EndOutput}
\algnotext{EndOutput}
\algblock{Result}{EndResult}
\algnotext{EndResult}
\newcommand{\Desc}[2]{\State \makebox[5em][l]{#1}#2}
\def\shor{{\sc Shor}}

\makeatletter
\newcommand{\StatexIndent}[1][3]{%
  \setlength\@tempdima{\algorithmicindent}%
  \Statex\hskip\dimexpr#1\@tempdima\relax}
\makeatother

\usepackage{amsmath}
\usepackage[colorlinks=true,linkcolor=blue,citecolor=dgreen]{hyperref}
\usepackage[shortlabels]{enumitem}
\usepackage{amssymb}
\usepackage{amsmath,float}
\usepackage{tikz}
\usepackage{cleveref}
\usepackage[numbers]{natbib}
\usepackage[english]{babel} 
\usepackage{blindtext}
 \usetikzlibrary{patterns}

\definecolor{shadecolor}{rgb}{0.9,0.9,0.9}
\def\={\equiv}

\newtheorem{theorem}{Theorem}

\newtheorem{thm}[theorem]{Theorem}

\newtheorem{conj}[theorem]{Conjecture}
\newtheorem{ques}{Question}

\newtheorem{lemma}[theorem]{Lemma}

\theoremstyle{definition}

\newcommand{\Z}{\mathbb{Z}}
\newcommand{\GG}{\mathcal{G}}
\newcommand{\HH}{\mathcal{H}}

\newcommand{\XX}{\mathcal{X}}
\newcommand{\YY}{\mathcal{Y}}

\definecolor{dgreen}{rgb}{0,.8,0}
\renewcommand{\leq}{\leqslant}
\renewcommand{\geq}{\geqslant}
\renewcommand{\le}{\leqslant}
\renewcommand{\ge}{\geqslant}
\renewcommand{\emptyset}{\varnothing}

\definecolor{SkyBlue}{RGB}{86,180,233}

\def\eref#1{$(\ref{#1})$}
\def\sref#1{\S$\ref{#1}$}
\def\lref#1{Lemma~$\ref{#1}$}
\def\lnref#1{Line~$\ref{#1}$}
\def\tref#1{Theorem~$\ref{#1}$}
\def\cjref#1{Conjecture~$\ref{#1}$}

\def\aref#1{Procedure~$\ref{#1}$}
\def\qref#1{Question~$\ref{#1}$}
\providecommand{\LatinGrid}[1]{%
  \begin{scope}[x=1cm,y=1cm]
    \draw[thick] (0,0) rectangle (#1,-#1);
    \draw[gray!60] (0,0) grid (#1,-#1);
  \end{scope}%
}
\crefname{equation}{}{}
\crefname{enumi}{}{}
\crefmultiformat{lemma}{lemmas~#2#1#3}{ and~#2#1#3}{, #2#1#3}{ and~#2#1#3}

\title{\bf Latin Squares whose transversals intersect in unusual ways}

\author{Afsane Ghafari and Ian M. Wanless}

\thanks{
School of Mathematics, Monash University, Vic 3800, Australia.
{\tt afsane.ghafaribaghestani@monash.edu, ian.wanless@monash.edu}.
}

\date{}
\begin{document}

\maketitle

\begin{abstract}
  A \emph{latin square} of order $n$ is an $n\times n$ array in which
  each of $n$ symbols occurs exactly once in each row and column.  A
  \emph{transversal} in such a square is a selection of $n$ entries
  that includes one representative of each row and column, and one of
  each symbol.  For all even orders $n\ge 28$ except $n=30$, we
  construct a latin square of order $n$ in which every pair of
  transversals share at least one entry. We conjecture that in our
  squares there is no single entry that is common to all transversals.
  We prove this conjecture for $n\le10\,000$ by finding transversals
  using an algorithm that is likely to be of independent interest.

  We say that a transversal is \emph{dominant} if it intersects every
  other transversal of the same latin square. We show that there exist
  latin squares of order $n$ that have a dominant transversal for
  $n\in\{5,7\}$ and also for all $n\ge8$ such that
  $n\not\equiv3\bmod4$.
\end{abstract}

\section{Introduction}\label{s:intro}

A \emph{latin square} is an $n \times n$ array consisting of $n$
distinct symbols where each symbol appears exactly once in each row
and each column. Let $L$ be a latin square over a symbol set $S$ of
size $n$.
A triple $(r,c,s)$ is called an \emph{entry} of a latin square $L$ when
$s$ is the symbol in cell $(r,c)$ of $L$.
Note that the latin property ensures that two
different entries agree in at most one coordinate. A \emph{diagonal}
$D$ of $L$ is a set of entries of $L$ that contains exactly one
representative from each row and each column. The weight of diagonal
$D$ is defined to be the number of distinct symbols included in the
entries in $D$. A \emph{transversal} of a latin square of order $n$ is
a diagonal of weight $n$. We say that an entry of a latin square is
\emph{pinned} if it is contained in every transversal (and there is at
least one transversal).  For details of the fascinating history of
transversals, see the survey articles \cite{Mon24,Wan11}.

Available evidence suggests that transversals behave differently
depending on the parity of the order $n$ of the latin square.  Latin
squares of odd order are conjectured to all have transversals, whereas
we know of superexponentially many latin squares of even order that
have no transversals \cite{CW17}. Even when latin squares of even
order have transversals, they can be heavily constricted, as evidenced
by this recent result from \cite{GW26}.

\begin{thm}\label{t:pinned}
  Let $n\geq 10$ be an even integer. There exists a latin square of
  order $n$ with at least $\big\lfloor {n}/{6} \big\rfloor$ 
  pinned entries.
\end{thm}

Distinct transversals of a latin square of order $n$ can share any
number up to $n-3$ entries in common \cite{CW10}.  The original
motivation for studying transversals came from orthogonal
latin squares. In that context, it is crucial to study disjoint
transversals, that is, ones that have no entries in common. From the
following recent breakthrough result of Bowtell and Montgomery
\cite{BM26} we understand the typical behaviour of transversals in
this regard:

\begin{thm}\label{t:mostOLS}
  Asymptotically almost all latin squares have a decomposition into disjoint
  transversals.
\end{thm}

The present paper is motivated by trying to understand transversal
behaviour which is very far from the typical behaviour seen in
\tref{t:mostOLS}. The latin squares in \tref{t:pinned} are in some
sense very far from having disjoint transversals. In this paper we
demonstrate a different way in which latin squares of even order can
fail to have disjoint transversals.

\begin{conj}\label{cj:main}
  Let $n \ge 28$ be an even integer. There is a latin square of order
  $n$ that has no two disjoint transversals, but also has no pinned entry.
\end{conj}

The existence of the latin squares that are the subject of
\cjref{cj:main} was posed to us as a question by Nick Wormald.  We
will prove \cjref{cj:main} for $n=28$ and $32\le n\le 10\,000$.  We
believe the construction that we use will generalise to all even
$n>10\,000$ but we have not been able to demonstrate existence of the
transversals required to show that no entry is pinned. We can,
however, show that they will never have disjoint transversals.

\medskip

In a related but distinct direction we also investigate latin squares
that have a \emph{dominant transversal}, which we define to be
a transversal that intersects all other transversals. We show that
some nice classes of latin squares do not have dominant transversals.
The techniques used to prove \tref{t:mostOLS} can be adapted to show that
almost all latin squares do not have a dominant transversal. So our last main
result can also be viewed as exhibiting a pathological behaviour:

\begin{thm}\label{t:dom}
  If $n>6$ and $n\not\equiv3\bmod4$ then there exists a latin square
  of order $n$ that has a dominant transversal.
\end{thm}

It is plausible that \tref{t:dom} remains true without the
restriction that $n\not\equiv3\bmod4$.  Similarly, \cjref{cj:main} may
be true for some values of $n<28$, although we know from \cite{EW12}
that it is not true for $n\le8$.

The structure of the paper is as follows. In \sref{s:algo} we describe
a randomised algorithm which is effective at finding transversals. It
can have restrictions imposed as to whether certain entries are or are
not included.  In \sref{s:nodisj}, we give a construction for latin
squares which we believe will prove \cjref{cj:main}. We are able to
show that our latin squares have some of the required properties, and
to confirm they have all of the required properties when $n\le
10\,000$. In \sref{s:mopup} we give a couple more examples which plug
gaps for small orders not covered by \sref{s:nodisj}. In \sref{s:dom}
we discuss dominant transversals and prove \tref{t:dom}.

We assume throughout that our latin squares of order $n$ are indexed
by $\Z_n$, meaning that we use $\Z_n$ for row indices, column indices
and symbols.

\section{Algorithm Explanation}\label{s:algo}

In this section, we describe a hill-climbing algorithm that quickly
finds transversals in even moderately large latin squares.  Suppose
that we are looking for a transversal in a given latin square $L$ of
order $n$.  The user is able to specify a set $F$ of forbidden entries
and a set $R$ of required entries. The task for the algorithm is to
find a transversal $T$ of $L$ that includes $R$ and is disjoint from
$F$. The algorithm progressively modifies a permutation $P$ of
$\Z_n$. At any given time $P$ induces a diagonal $D_P$ defined by
$D_P=\{(i,P[i],\star):i\in\Z_n\}$. Here, and henceforth we use $\star$
to denote the unique symbol that makes a triple into an entry of $L$
(so in this particular case $\star$ denotes $L[i,P[i]]$).  The
\emph{weight} $w(P)$ of $P$ is defined to be the number of different
symbols on $D_P$.  In other words,
\[
w(P)=\big|\{s:(i,P[i],s)\in D_P\text{ for some }i\in\Z_n\}\big|.
\]
The goal is to increase the weight of $P$ until $w(P)=n$, at which
point we have located a transversal. The permutation $P$ is modified
using Procedure~\ref{a:Shor}, which performs \#-moves as introduced by Shor~\cite{Sho82}.

\begin{algorithm}[H]
  \caption{\label{a:Shor}Perform a Shor \#-move}
  \begin{algorithmic}[1]  
    \Input
    \Desc{$i,j$}{Indices in $\Z_n$}
    \EndInput
    \Result
    \Desc{Swaps the values of $P[i]$ and $P[j]$, if allowed}
    \EndResult
    \Procedure{Shor$(i,j)$}{}
    \If{$\{(i,P[j],\star),(j,P[i],\star)\}\cap F=\emptyset$ and $\{(i,P[i],\star),(j,P[j],\star)\}\cap R=\emptyset$}
    \State $t:=P[i]$.
    \State $P[i]:=P[j]$.
    \State $P[j]:=t$.
    \State Update $P^{-1}$, $E$, $U$.\label{ln:upd}
    \EndIf
    \EndProcedure
  \end{algorithmic}
\end{algorithm}

Our algorithm relies on the following variables and data structures:
\begin{itemize}
  \item A permutation $P$ of $\Z_n$. We consider $P$ to be a map from
    rows to columns of $L$, thereby identifying a diagonal $D_P$ of $L$.
  \item The map $P^{-1}$ which is the inverse of $P$.
  \item The map $L^{-1}:\Z_n\times\Z_n\rightarrow\Z_n$ defined by
    $L^{-1}[r,s]=c$ if $L[r,c]=s$, which identifies the column in
    which a particular symbol occurs in a particular row. Since $L$
    does not change, $L^{-1}$ can be precomputed and stored in an
    array if desired.  Doing so is advisable unless $L$ is defined by
    some easily inverted function.
  \item A set $U$ of ``unused'' symbols. These are the symbols that
    currently do not appear on $D_P$.
  \item A set $E$ of ``excess'' rows. These are rows in which the
    symbol on $D_P$ also occurs on $D_P$ in at least one other row.
\end{itemize}
Other auxiliary variables may be required to efficiently perform steps such as
\lnref{ln:upd} of \aref{a:Shor}. For example, we found it helpful to
maintain, for each symbol, a list of the rows in which that symbol
occurs on $D_P$. However, for the sake of clarity, we omit most low-level
implementation details.

Our main algorithm is \aref{a:main} and line numbers refer to it
unless otherwise specified.  In \lnref{ln:randP}, we choose the
initial value of $P$. How we do this is not particularly important,
but $D_P$ must include every entry in $R$ and not include any entry in
$F$.  Subsequently, our algorithm will only ever change $P$ via its
many calls to \aref{a:Shor}. The effect of \aref{a:Shor} is to perform
a Shor-\# operation, that is, to swap two values of $P$.  If doing so
would lose any required entry or introduce any forbidden entry then
the procedure call has no effect. It follows that $D_P$ will always
include every entry in $R$ and avoid including any entry in $F$. In
particular, if the algorithm succeeds in returning a transversal, that
transversal will include $R$ and avoid $F$.

Typically the initial weight of $P$ will be far from $n$. However, we
now argue that \lnref{ln:wimp} will never reduce $w(P)$.  Firstly,
$L[r,P[r]]$ appears elsewhere on $D_P$ because $r\in E$, so we usually
do not mind losing it (there is an exception, which is that it will
occasionally be the case that in one move we lose the only two copies of
$L[r,P[r]]$). Let $c=L^{-1}[r,s]$.  We note that $L[r,c]=s\in U$
so including $(r,c,s)$ in $D_P$ means that we gain a symbol that was
previously unused.  Hence, $w(P)$ will never decrease, nor can it increase
by more than two. It will sometimes remain unchanged and sometimes 
increase by either 1 or 2 depending on factors such as whether
$P^{-1}[c]\in E$, and whether the symbol $L[P^{-1}[c],P[r]]$
appears elsewhere on $D_P$.  The hope is that by hill
climbing we will eventually exit from the {\bf while} loop that begins in
\lnref{ln:firstwhile}.  It may be prudent to incorporate a mechanism
to abandon this loop and return {\bf Fail} if progress has not been
made after sufficiently many steps. However, we did not need that for
our examples.

Once we have a diagonal of weight $w(P)\ge n-2$, our strategy changes.
Because $w(P)$ increases by at most 2 with each call to \aref{a:Shor},
the first time that $w(P)\ge n-2$ we will have $w(P)\in\{n-2,n-1\}$.
If $w(P)=n-1$ then it is impossible for a single Shor-\# move to reach
a transversal, so we accept a decrease to $w(P)=n-2$.  This situation
is handled by the {\bf while} loop that begins in \lnref{ln:2ndwhile}.
Inside that loop, the fact that we choose $r\in E$ ensures that our
call to \aref{a:Shor} either preserves $w(P)$ or reduces it by~1.
It follows that every time we reach \lnref{ln:rr'} we have
$w(P)=n-2$. In this circumstance we try every Shor-\# that has a
chance to move us to a transversal.  If the move does not produce a
transversal then we undo it by calling \aref{a:Shor} with the same
parameters. Of course, if we do obtain a transversal then our goal has
been achieved and we can stop.  In \lnref{ln:rr'}, $r$ and $r'$ should
be distinct, and it never helps to choose them to be the only two rows
where some symbol currently occurs on $D_P$. It could however be the
case that they both contain the same symbol on $D_P$ if that symbol
occurs three times on $D_P$.  If no move takes us to a transversal
then in \lnref{ln:sidestep} we make a random move that does not
decrease $w(P)$, and continue from there. The $^*$ in \lnref{ln:star}
indicates that it
is important here not to choose $r,r'$ to be the only two rows in
which some symbol appears on $D_P$, since otherwise we might reduce
$w(P)$.

The {\bf for} loop in \lnref{ln:for100} is designed to prevent
spending too long in the state of having $w(P)\ge n-2$. The reason for
this is that there are very few Shor-\# moves available when $w(P)\ge n-2$
and it is theoretically possible to get trapped in a situation
where all possibilities have been fully explored and none of them lead
to progress. A more sophisticated version of our algorithm might reset
the counter every time a row is used in a call to \aref{a:Shor} that
has not been used since we reached \lnref{ln:for100}.  As it stands,
if we exhaust the loop in \lnref{ln:for100} then we abandon all
progress and start again. A possible less drastic variant in this
situation would instead make some random Shor-\# moves to move away
from the present $P$, but still keep a relatively high value of
$w(P)$.

\begin{algorithm}[H]
  \caption{\label{a:main}Finds a transversal of a given latin square}
\begin{algorithmic}[1]
  \Input
  \Desc{$L$}{A latin square.}
  \Desc{$R$}{A set of required entries of $L$.}
  \Desc{$F$}{A set of forbidden entries of $L$.}
  \EndInput 
  \Output
  \Desc{$T$}{A transversal of $L$ that contains $R$ and is disjoint from $F$.}
  \EndOutput
  \Procedure{FindTrans$(L,R,F)$}{}
  \State Choose a random $P$ that includes $R$ and avoids $F$.\label{ln:randP}
  \State Initialise $E$, $U$.
  \While{$w(P)<n-2$}\label{ln:firstwhile}
  \State Choose a random $r\in E$.
  \State Choose a random $s\in U$.
  \State \shor$(r,P^{-1}[L^{-1}[r,s]])$.\label{ln:wimp}
  \EndWhile
  \For{$i\in\{1,\ldots,100\}$}\label{ln:for100}
  \While{$w(P)=n-1$}\label{ln:2ndwhile}
    \State Choose a random $r\in E$.
    \State Choose a random $r'\in\Z_n\setminus\{r\}$.
    \State \shor$(r,r')$.
  \EndWhile
  \For{each pair $\{r,r'\}\subseteq E$}\label{ln:rr'}
  \State \shor$(r,r')$.
  \If {$w(P)=n$}
  \State\Return{$P$}.
  \EndIf
  \State \shor$(r,r')$.
  \EndFor
  \State Choose a random$^*$ pair $\{r,r'\}\subseteq E$.\label{ln:star}
  \State \shor$(r,r')$.\label{ln:sidestep}
  \EndFor
  \State \Return{\sc FindTrans}$(L,R,F)$.
  \EndProcedure
  \end{algorithmic}
\end{algorithm}

When using \aref{a:main}, it is the user's responsibility to ensure
that it is plausible that there is a transversal meeting the
requirements.  In an extreme case, it may not even be possible to
choose an initial value of $P$ that satisfies the requirement to
include $R$ and avoid $F$. Even if that does not happen, putting a
large proportion of all entries into $R\cup F$ might limit the available
moves so much that the algorithm gets stuck. On the other hand, if the
user has information that certain entries cannot ever be included in a
transversal then feeding this information to the algorithm might avoid
it wasting a lot of time exploring dead-ends. In the next section we
will see examples of exactly this situation.

Consider an entry $(r,c,s)\in R$. We know that all other entries that
share a row, column or symbol with $R$ cannot end up in a tranversal
that includes $R$. So these entries can be included in $F$ if desired.
Including entries from row $r$ or column $c$ in $F$ will not have any
effect since \aref{a:Shor} would never switch away from $(r,c,s)$.
However, entries that have symbol $s$ are more interesting. It is
plausible to leave them out of $F$ to allow them to be used in
intermediate stages of the search, knowing that they will eventually
need to be switched out of $P$ before we find a transversal. However, it
is probably more efficient to include them in $F$.

\section{Pairwise-intersecting transversals with no universal intersection}\label{s:nodisj}

We start this section by introducing a tool called the $\Delta$-Lemma,
which has proved immensely useful for studying transversals
\cite{Wan11}.  Viewing a latin square $L$ indexed by $\Z_n$
as a set of triples $(r,c,s)\in\Z_n^3$, we define the
function $\Delta : L \rightarrow \Z$ by
\begin{align}\label{Copyfascinating}
\Delta (r,c,s) \equiv s-r-c \bmod n\text{ where } -\dfrac{n}{2} < \Delta(r,c,s) \leq \dfrac{n}{2}.
\end{align}

\begin{lemma}\label{l:Delta}
Let $L$ be a latin square of order $n$ indexed by $\Z_n$. If $T$ is a
transversal of $L$ then, modulo $n$,
\begin{equation}\label{e:Deltaeq}
\sum_{(r,c,s)\in T} \Delta(r,c,s) \equiv
\begin{cases}
	0  & \text{if $n$ is odd},\\
	\frac{1}{2} n & \text{if $n$ is even}.
\end{cases}
\end{equation}
\end{lemma}

Our latin squares are indexed by $\Z_n$ and by default calculations
involving their rows, columns and symbols will be modulo
$n$. Nevertheless, we stress that we have chosen to make $\Delta$ an
integer valued function so that we can use inequalities to contradict
\eref{e:Deltaeq}.

\begin{theorem}\label{t:surrogate}
  For $n=28$ and each even integer $n$ in the interval $[32,10000]$,
  there exists a latin square of order $n$ such that every pair of its
  transversals shares at least one common entry, yet no single entry
  is shared by all transversals.  
\end{theorem} 

We prove \Cref{t:surrogate} by dividing the argument into two main
cases, depending on whether $n\equiv 0\ \text{or}\ 2\bmod 4$. The
exceptional cases $n\in\{28,32,34\}$ will be handled separately.

We begin by considering the case $n= 4k + 2$, where $k\geq4$
is an integer. Let $\GG_n$ be the latin square of order $n$ defined by,  
\begin{equation}\label{Structure}
\GG_n[a,b] = \begin{cases}
    a+b+4 &  \text{if $a\in \{0,5,10\}$, and $b\equiv 1$ mod 4 and $b>a+1$}, \\
    a+b+4& \text{if $(a,b) \in \{(5,3),(10,3),(10,7)\}$},\\
    a+b+3& \text{if $(a,b) \in \{(1,2),(6,6),(11,10)\}$},\\
    a+b+1 &
     \begin{aligned}[t]
     & \text{if } (a,b) \in \{(0,3),(0,4),(5,7),(5,8),(10,11),(10,12),(15,12),\\
     & \qquad \qquad \, \, \,(15,13),(16,12)\},
     \end{aligned}\\
     a+b-1 &
    \begin{aligned}[t]
     & \text{if } (a,b) \in \{(1,3),(1,4),(1,5),(6,7),(6,8),(6,9),(11,11),(11,12),\\
     & \qquad \qquad \,\, \,(11,13),(16,13)\},
     \end{aligned}\\
    a+b-3&\text{if $(a,b) \in \{(4,2),(4,5),(9,6),(9,9),(14,10),(14,13)\}$},\\
    a+b-4&\text{if $a\in \{4,9,14\}$ and $b \equiv 1$ mod 4 and $b>a+1$}, \\ 
    a+b-4&\text{if $(a,b) \in \{(9,3),(14,3),(14,7)\}$},\\
    a+b+2&\text{if $a=15$ and $12\ne b\equiv 0$ mod 2,}\\
     a+b-2&\text{if $a=17$ and $b\equiv 0$ mod 2,}\\
    a+b+2&  \text{if $18 \le a < 3k-9 , a \equiv 0 $ mod 3 and $b \equiv 0$ mod 2,}\\
    a+b-2& \text{if $18 \le a < 3k-9 , a \equiv 2 $ mod 3 and $b \equiv 0$ mod 2,}\\
    a+b & \text{otherwise.}
  \end{cases}
\end{equation}
\begin{lemma}\label{twomodfour}
  Let $n=4k+2$ for an integer $k\ge 9$.  The latin
  square $\GG_n$ contains no pair of disjoint transversals.
\end{lemma}

\begin{proof}
From \Cref{Structure} we see that $-4\le\Delta(r,c,s)\le4$ for all
$(r,c,s)$ in $\GG_n$. Our aim is to show that every
transversal of $\GG_n$ has to include at least two of the entries
\begin{align}\label{Deltavalue3}
 (1,2,6),\ (6,6,15),\ (11,10,24), 
\end{align}
which are the only entries with $\Delta$-value equal to $3$. We claim
that any selection of entries with at most one of the entries in
\Cref{Deltavalue3} cannot form a transversal. Note that the
maximum sum of the $\Delta$ values of such a selection is
\begin{align}\label{ImpCon}
  3\times4+3+ 1+ \sum_{i=1}^{k-8} 2 = 2k
  < 2k+1 = \frac{n}{2},
\end{align}
where $k-8$ is the number of rows containing an entry
whose $\Delta$-value is $2$.   
Also, the sum of the minimum $\Delta$-value from each row is
\begin{align}\label{e:minG}
  3\times(-1-4)-1+\sum_{i=1}^{k-8} (-2)= -2k > -2k-1=-\dfrac{n}{2}.
\end{align}
It follows from \eref{ImpCon} and \eref{e:minG} and \lref{l:Delta}
that every transversal contains at least two of the entries
in \Cref{Deltavalue3}, so no pair of transversals in $\GG_n$ is disjoint.
\end{proof}
It is important to note that the condition $k\ge 9$ in
\Cref{twomodfour} is crucial, as the argument fails for values $4\le k\le 8$.
Furthermore, for every $k$ in this range, the latin square
arising from the construction in \eqref{Structure} admits at least two
disjoint transversals.  In particular, the latin square $\GG_{18}$
exhibits two disjoint transversals, shown below in different
colours. The shaded entries indicate the entries that differ
from the addition table of $\Z_n$.

\begin{equation*}
\vcenter{
\hbox{
\begin{tikzpicture}[scale=0.7]
  \definecolor{SQa}{RGB}{0,114,188}   
  \definecolor{SQb}{RGB}{140,198,62}  
  \def\fillop{0.45} 

  \def\Ldata{%
  0,1,2,4,5,9,6,7,8,13,10,11,12,17,14,15,16,3,
  1,2,6,3,4,5,7,8,9,10,11,12,13,14,15,16,17,0,
  2,3,4,5,6,7,8,9,10,11,12,13,14,15,16,17,0,1,
  3,4,5,6,7,8,9,10,11,12,13,14,15,16,17,0,1,2,
  4,5,3,7,8,6,10,11,12,9,14,15,16,13,0,1,2,17,
  5,6,7,12,9,10,11,13,14,0,15,16,17,4,1,2,3,8,
  6,7,8,9,10,11,15,12,13,14,16,17,0,1,2,3,4,5,
  7,8,9,10,11,12,13,14,15,16,17,0,1,2,3,4,5,6,
  8,9,10,11,12,13,14,15,16,17,0,1,2,3,4,5,6,7,
  9,10,11,8,13,14,12,16,17,15,1,2,3,0,5,6,7,4,
  10,11,12,17,14,15,16,3,0,1,2,4,5,9,6,7,8,13,
  11,12,13,14,15,16,17,0,1,2,6,3,4,5,7,8,9,10,
  12,13,14,15,16,17,0,1,2,3,4,5,6,7,8,9,10,11,
  13,14,15,16,17,0,1,2,3,4,5,6,7,8,9,10,11,12,
  14,15,16,13,0,1,2,17,4,5,3,7,8,6,10,11,12,9,
  17,16,1,0,3,2,5,4,7,6,9,8,10,11,13,12,15,14,
  16,17,0,1,2,3,4,5,6,7,8,9,11,10,12,13,14,15,
  15,0,17,2,1,4,3,6,5,8,7,10,9,12,11,14,13,16}

  \newcommand{\EntryAt}[2]{%
    \pgfmathtruncatemacro{\idx}{#1*18 + #2 + 1}%
    \begingroup\def\val{}%
    \pgfmathtruncatemacro{\kk}{0}%
    \foreach \x [count=\kk] in \Ldata {%
      \ifnum\kk=\idx\relax \xdef\val{\x}\fi
    }%
    \val\endgroup
  }

  \foreach [count=\i from 0] \j in {14,6,10,0,7,1,4,15,9,13,12,5,8,2,17,11,3,16}{
    \fill[SQa, opacity=\fillop] (\j,-\i) rectangle ++(1,-1);
  }
  \foreach [count=\i from 0] \j in {1,8,0,3,11,13,9,5,15,2,7,6,16,12,4,14,10,17}{
    \fill[SQb, opacity=\fillop] (\j,-\i) rectangle ++(1,-1);
  }

  \fill[pattern=north east lines,pattern color=black] (3,-0) rectangle ++(1,-1);
  \fill[pattern=north east lines,pattern color=black] (4,-0) rectangle ++(1,-1);
  \fill[pattern=north east lines,pattern color=black] (5,-0) rectangle ++(1,-1);
  \fill[pattern=north east lines,pattern color=black] (9,-0) rectangle ++(1,-1);
  \fill[pattern=north east lines,pattern color=black] (13,-0) rectangle ++(1,-1);
  \fill[pattern=north east lines,pattern color=black] (17,-0) rectangle ++(1,-1);
  \fill[pattern=north east lines,pattern color=black] (2,-1) rectangle ++(1,-1);
  \fill[pattern=north east lines,pattern color=black] (3,-1) rectangle ++(1,-1);
  \fill[pattern=north east lines,pattern color=black] (4,-1) rectangle ++(1,-1);
  \fill[pattern=north east lines,pattern color=black] (5,-1) rectangle ++(1,-1);
  \fill[pattern=north east lines,pattern color=black] (2,-4) rectangle ++(1,-1);
  \fill[pattern=north east lines,pattern color=black] (5,-4) rectangle ++(1,-1);
  \fill[pattern=north east lines,pattern color=black] (9,-4) rectangle ++(1,-1);
  \fill[pattern=north east lines,pattern color=black] (13,-4) rectangle ++(1,-1);
  \fill[pattern=north east lines,pattern color=black] (17,-4) rectangle ++(1,-1);
  \fill[pattern=north east lines,pattern color=black] (3,-5) rectangle ++(1,-1);
  \fill[pattern=north east lines,pattern color=black] (7,-5) rectangle ++(1,-1);
  \fill[pattern=north east lines,pattern color=black] (8,-5) rectangle ++(1,-1);
  \fill[pattern=north east lines,pattern color=black] (9,-5) rectangle ++(1,-1);
  \fill[pattern=north east lines,pattern color=black] (13,-5) rectangle ++(1,-1);
  \fill[pattern=north east lines,pattern color=black] (17,-5) rectangle ++(1,-1);
  \fill[pattern=north east lines,pattern color=black] (6,-6) rectangle ++(1,-1);
  \fill[pattern=north east lines,pattern color=black] (7,-6) rectangle ++(1,-1);
  \fill[pattern=north east lines,pattern color=black] (8,-6) rectangle ++(1,-1);
  \fill[pattern=north east lines,pattern color=black] (9,-6) rectangle ++(1,-1);
  \fill[pattern=north east lines,pattern color=black] (3,-9) rectangle ++(1,-1);
  \fill[pattern=north east lines,pattern color=black] (6,-9) rectangle ++(1,-1);
  \fill[pattern=north east lines,pattern color=black] (9,-9) rectangle ++(1,-1);
  \fill[pattern=north east lines,pattern color=black] (13,-9) rectangle ++(1,-1);
  \fill[pattern=north east lines,pattern color=black] (17,-9) rectangle ++(1,-1);
  \fill[pattern=north east lines,pattern color=black] (3,-10) rectangle ++(1,-1);
  \fill[pattern=north east lines,pattern color=black] (7,-10) rectangle ++(1,-1);
  \fill[pattern=north east lines,pattern color=black] (11,-10) rectangle ++(1,-1);
  \fill[pattern=north east lines,pattern color=black] (12,-10) rectangle ++(1,-1);
  \fill[pattern=north east lines,pattern color=black] (13,-10) rectangle ++(1,-1);
  \fill[pattern=north east lines,pattern color=black] (17,-10) rectangle ++(1,-1);
  \fill[pattern=north east lines,pattern color=black] (10,-11) rectangle ++(1,-1);
  \fill[pattern=north east lines,pattern color=black] (11,-11) rectangle ++(1,-1);
  \fill[pattern=north east lines,pattern color=black] (12,-11) rectangle ++(1,-1);
  \fill[pattern=north east lines,pattern color=black] (13,-11) rectangle ++(1,-1);
  \fill[pattern=north east lines,pattern color=black] (3,-14) rectangle ++(1,-1);
  \fill[pattern=north east lines,pattern color=black] (7,-14) rectangle ++(1,-1);
  \fill[pattern=north east lines,pattern color=black] (10,-14) rectangle ++(1,-1);
  \fill[pattern=north east lines,pattern color=black] (13,-14) rectangle ++(1,-1);
  \fill[pattern=north east lines,pattern color=black] (17,-14) rectangle ++(1,-1);
  \fill[pattern=north east lines,pattern color=black] (0,-15) rectangle ++(1,-1);
  \fill[pattern=north east lines,pattern color=black] (2,-15) rectangle ++(1,-1);
  \fill[pattern=north east lines,pattern color=black] (4,-15) rectangle ++(1,-1);
  \fill[pattern=north east lines,pattern color=black] (6,-15) rectangle ++(1,-1);
  \fill[pattern=north east lines,pattern color=black] (8,-15) rectangle ++(1,-1);
  \fill[pattern=north east lines,pattern color=black] (10,-15) rectangle ++(1,-1);
  \fill[pattern=north east lines,pattern color=black] (12,-15) rectangle ++(1,-1);
  \fill[pattern=north east lines,pattern color=black] (13,-15) rectangle ++(1,-1);
  \fill[pattern=north east lines,pattern color=black] (14,-15) rectangle ++(1,-1);
  \fill[pattern=north east lines,pattern color=black] (16,-15) rectangle ++(1,-1);
  \fill[pattern=north east lines,pattern color=black] (12,-16) rectangle ++(1,-1);
  \fill[pattern=north east lines,pattern color=black] (13,-16) rectangle ++(1,-1);
  \fill[pattern=north east lines,pattern color=black] (0,-17) rectangle ++(1,-1);
  \fill[pattern=north east lines,pattern color=black] (2,-17) rectangle ++(1,-1);
  \fill[pattern=north east lines,pattern color=black] (4,-17) rectangle ++(1,-1);
  \fill[pattern=north east lines,pattern color=black] (6,-17) rectangle ++(1,-1);
  \fill[pattern=north east lines,pattern color=black] (8,-17) rectangle ++(1,-1);
  \fill[pattern=north east lines,pattern color=black] (10,-17) rectangle ++(1,-1);
  \fill[pattern=north east lines,pattern color=black] (12,-17) rectangle ++(1,-1);
  \fill[pattern=north east lines,pattern color=black] (14,-17) rectangle ++(1,-1);
  \fill[pattern=north east lines,pattern color=black] (16,-17) rectangle ++(1,-1);
  \foreach \i in {0,...,17}{
    \foreach \j in {0,...,17}{
      \node[font=\small] at (\j+0.5,-\i-0.5) {\EntryAt{\i}{\j}};
    }
  }

  \LatinGrid{18}
\end{tikzpicture}
}
}
\end{equation*}

\vspace*{1cm}

Next, for each order $n\in \{22,26,30,34\}$ we exhibit two
transversals $T^{(1)}_n$ and $T^{(2)}_n$ where
$T^{(1)}_n\cap T^{(2)}_n=\emptyset$.

{\small
\begin{align*}
T^{(1)}_{22} = \{& (0,11,11), (1,17,18), (2,2,4), (3,18,21), (4,5,6), (5,14,19), (6,7,12), (7,20,5), (8,15,1), \\
&(9,9,15), (10,0,10), (11,6,17), (12,13,2), (13,16,7), (14,3,16), (15,1,16), (16,8,2),\\
& (17,19,14), (18,4,0), (19,12,9), (20,10,8), (21,21,20) \},\\[1ex]
T^{(2)}_{22} = \{& (0,21,3), (1,10,11), (2,0,2), (3,5,8), (4,19,1), (5,17,4), (6,15,21), (7,7,14), (8,4,12),\\
& (9,8,17), (10,9,19), (11,16,5), (12,6,18), (13,2,15), (14,14,6), (15,18,11), (16,12,4),  \\
&(17,3,20), (18,20,16), (19,13,10), (20,11,9), (21,1,0)\},\\[1ex]
\end{align*}
}
{\small
\begin{align*}
T^{(1)}_{26} = \{& (0,21,25), (1,22,23), (2,24,0), (3,19,22), (4,9,9), (5,3,12), (6,6,15), (7,23,4), \\
&(8,11,19), (9,15,24), (10,7,21), (11,25,10), (12,16,2), (13,4,17), (14,0,14), (15,14,5), \\
& (16,18,8), (17,1,18), (18,2,20), (19,20,13), (20,12,6), (21,8,3), (22,5,1), (23,10,7), \\
&(24,13,11), (25,17,16) \},\\[1ex]
T^{(2)}_{26} = \{& (0,9,13), (1,2,6), (2,0,2), (3,8,11), (4,4,8), (5,11,16), (6,18,24), (7,3,10), (8,12,20),\\
& (9,20,3), (10,17,0), (11,14,23), (12,15,1), (13,13,0), (14,5,19), (15,16,7), (16,7,22), \\
& (17,23,15), (18,25,16), (19,19,12), (20,21,15), (21,1,22), (22,22,18), (23,24,20),\\
& (24,6,4), (25,10,9) \},\\[1ex]
T^{(1)}_{30} = \{& (0,29,3), (1,12,13), (2,4,6), (3,28,1), (4,11,15), (5,21,0), (6,1,7), (7,2,9), (8,18,26),\\
& (9,16,25), (10,7,21), (11,10,24), (12,26,8), (13,9,22), (14,0,14), (15,25,11),   \\
&(16,19,4), (17,15,2), (18,5,23), (19,8,27), (20,27,14), (21,13,4), (22,20,10), (23,6,5), \\
& (24,24,19), (25,3,28), (26,23,19), (27,14,11), (28,22,20), (29,17,16) \},\\[1ex]
T^{(2)}_{30} = \{& (0,14,14), (1,3,3), (2,21,23), (3,13,16), (4,9,9), (5,1,6), (6,15,21), (7,0,7), (8,23,0),\\
& (9,17,22), (10,16,26), (11,24,5), (12,8,20), (13,2,15), (14,7,17), (15,27,13),\\
&  (16,18,4), (17,4,22), (18,22,10), (19,10,29), (20,12,2), (21,20,11), (22,6,28),   \\
&(23,25,20), (24,19,16), (25,5,0), (26,28,24), (27,11,8), (28,29,27), (29,26,25)\},\\[1ex]
T^{(1)}_{34} = \{& (0,5,9), (1,29,30), (2,23,25), (3,19,22), (4,14,18), (5,33,8), (6,6,15), (7,13,20), \\
&(8,20,26), (9,30,5), (10,3,17), (11,25,1), (12,9,21), (13,21,0), (14,32,12), (15,24,7),  \\
&(16,8,24),(17,31,14), (18,15,32), (19,10,30), (20,11,31), (21,17,4), (22,28,19), \\
&  (23,22,5), (24,16,32),(25,1,26), (26,27,8), (27,0,27), (28,4,10), (29,18,1), (30,7,11),  \\
&(31,26,23), (32,12,10), (33,2,1) \},\\[1ex]
T^{(2)}_{34} = \{& (0,13,17), (1,12,13), (2,9,11), (3,4,7), (4,11,15), (5,3,12), (6,16,22), (7,29,2), \\
& (8,21,27), (9,14,18), (10,25,5), (11,10,20), (12,28,6), (13,17,31), (14,23,2), (15,26,8),  \\
& (16,18,0),(17,33,15), (18,24,8), (19,19,3), (20,1,21), (21,7,28), (22,32,10), (23,8,23), \\
&  (24,2,26), (25,27,7), (26,22,14), (27,6,33), (28,31,16), (29,15,11), (30,5,1), \\
& (31,30,29), (32,0,32), (33,20,19) \}.
\end{align*}
}

Suppose that $38\le n=4k+2\le 10000$. Using \aref{a:main} from the
previous section, we found that the latin square $\GG_n$ satisfies
\tref{t:surrogate}. To do this, we needed to find four transversals in
$\GG_n$. For each of the entries in \Cref{Deltavalue3} we found a
transversal that missed the entry in question but included the other
two entries in \Cref{Deltavalue3}. This ensured that no single entry
from \Cref{Deltavalue3} was common to all transversals. However, a
calculation similar to \Cref{ImpCon} implies that all transversals
that contain exactly two of the three entries from \Cref{Deltavalue3}
must also include the entry $(16,12,29)$ in order to satisfy
\lref{l:Delta}. To address this issue, we found a fourth transversal
that avoids the entry $(16,12,29)$ but includes all three entries from
\Cref{Deltavalue3}. Each time we searched for a transversal we
therefore had some entries we wanted to include in the transversal and
some that we definitely wanted to exclude. These requirements could be
handled using the parameters $R$ and $F$ for \aref{a:main}. In all of
our searches, in every row of $\GG_n$ that contained any entry with
nonzero $\Delta$-value, there were entries we could safely put in
$F$. For each row we decided which $\Delta$-value we would use from
that row and put all entries not having that value into $F$.  Doing
this ensured that the algorithm only searched among entries that
satisfied \lref{l:Delta}. This represented a massive speed up compared
to not having any restriction, because it guided the algorithm away
from choices that we knew could not work.

\bigskip

We next introduce the other family of our latin squares which cover the
case where $n\equiv 0\bmod4$.
Let $n=4k$, for an integer $k \geq 4$.
Consider the latin square $\HH_n$ given by
\begin{equation}\label{Structuretwo}
  \HH_n[a,b] = \begin{cases}
    a+b+4 &  \text{if $a\in \{0,5,10\}$ and $b\equiv 1$ mod 4 and $b-4a/5\ne1$}, \\ 
    a+b+3& \text{if $(a,b) \in \{(1,1),(6,5),(11,9)\}$},\\
    a+b+1& \text{if $a\in \{0,5,10\}$ and $1\le b-4a/5\le4$}, \\
    a+b-1&\text{if $a\in \{1,6,11\}$ and $2\le b-4(a-1)/5\le4$}, \\
    a+b-4&\text{if $a\in \{4,9,14\}$ and $b \equiv 1$ mod 4},\\
    a+b+2&  \text{if $15 \le a < 4k-21, a\equiv 3$ mod 4 and $b \equiv 0$ mod 2,}\\
    a+b-2& \text{if $15 \le a < 4k-21, a\equiv 1$ mod 4 and $b \equiv 0$ mod 2,}\\
    a+b & \text{otherwise.}
  \end{cases}
\end{equation}

\begin{lemma}\label{zeromodfour}
  Let $n=4k$ for $k\ge9$. The latin square $\HH_n$ does not have
  any pair of disjoint transversals.
\end{lemma}

\begin{proof}
From \Cref{Structuretwo} we see that $-4\le\Delta(r,c,s)\le4$ for all
entries $(r,c,s)$ in $\HH_n$.
Our aim is to show that every transversal must include exactly two of
the entries
\begin{align}\label{Delta3}
 (1,1,5),\ (6,5,14),\ (11,9,23), 
\end{align}
which are the only entries with $\Delta$-value equal to $3$. We claim
that any selection of entries with at most one of the entries in
\Cref{Delta3} cannot form a transversal. The reason is that the
maximum sum of the $\Delta$ values of such a selection is
\begin{equation}\label{e:maxH}
  3\times4+3+\sum_{i=1}^{k-9} 2= 2k - 3< 2k = \dfrac{n}{2},
\end{equation}
where $k-9$ is the number of rows containing an entry
with $\Delta$-value equal to $2$.
Also, the sum of the minimum $\Delta$-value from each row is
\begin{equation}\label{e:minH}
  3 \times (-1-4) +\sum_{i=1}^{k-9} (-2)= -2k+3 > -2k = - \dfrac{n}{2}.
\end{equation}
It follows from \eref{e:maxH} and \eref{e:minH} and \lref{l:Delta}
that every transversal contains at least two of the entries
in \Cref{Delta3}, so no pair of transversals in $\HH_n$ is disjoint.
\end{proof}

It should be emphasised that the condition $k\ge9$ in
\Cref{zeromodfour} is necessary, since the argument fails for $4\le
k\le 8$.  Notably, $\HH_{32}$ contains no transversals.
For every $k$ within the range $4\le k \le 7$, the latin square
obtained from the construction in \eqref{Structuretwo} contains at
least two disjoint transversals. As an illustration, the latin square
$\HH_{16}$, includes two disjoint transversals, shown below in
different colours. The shaded cells illustrates the entries that
differ from the addition table of $\Z_n$.
\providecommand{\LatinGrid}[1]{%
  \begin{scope}[x=1cm,y=1cm]
    \draw[thick] (0,0) rectangle (#1,-#1);
    \draw[gray!60] (0,0) grid (#1,-#1);
  \end{scope}%
}

\begin{equation*}
\vcenter{
\hbox{
\begin{tikzpicture}[scale=0.7]
  \definecolor{SQa}{RGB}{0,114,188}  
  \definecolor{SQb}{RGB}{140,198,62} 
  \def\fillop{0.45} 

  \def\Ldata{%
  0,2,3,4,5,9,6,7,8,13,10,11,12,1,14,15,
  1,5,2,3,4,6,7,8,9,10,11,12,13,14,15,0,
  2,3,4,5,6,7,8,9,10,11,12,13,14,15,0,1,
  3,4,5,6,7,8,9,10,11,12,13,14,15,0,1,2,
  4,1,6,7,8,5,10,11,12,9,14,15,0,13,2,3,
  5,10,7,8,9,11,12,13,14,2,15,0,1,6,3,4,
  6,7,8,9,10,14,11,12,13,15,0,1,2,3,4,5,
  7,8,9,10,11,12,13,14,15,0,1,2,3,4,5,6,
  8,9,10,11,12,13,14,15,0,1,2,3,4,5,6,7,
  9,6,11,12,13,10,15,0,1,14,3,4,5,2,7,8,
  10,15,12,13,14,3,0,1,2,4,5,6,7,11,8,9,
  11,12,13,14,15,0,1,2,3,7,4,5,6,8,9,10,
  12,13,14,15,0,1,2,3,4,5,6,7,8,9,10,11,
  13,14,15,0,1,2,3,4,5,6,7,8,9,10,11,12,
  14,11,0,1,2,15,4,5,6,3,8,9,10,7,12,13,
  15,0,1,2,3,4,5,6,7,8,9,10,11,12,13,14}

  \newcommand{\EntryAt}[2]{%
    \pgfmathtruncatemacro{\idx}{#1*16 + #2 + 1}%
    \begingroup\def\val{}%
    \pgfmathtruncatemacro{\kk}{0}%
    \foreach \x [count=\kk] in \Ldata {%
      \ifnum\kk=\idx\relax \xdef\val{\x}\fi
    }%
    \val\endgroup
  }

  \foreach [count=\i from 0] \j in {5,3,6,8,11,1,12,14,15,7,9,2,10,4,0,13}{
    \fill[SQa, opacity=\fillop] (\j,-\i) rectangle ++(1,-1);
  }
  \foreach [count=\i from 0] \j in {12,6,1,10,0,9,5,3,8,15,11,4,13,14,7,2}{
    \fill[SQb, opacity=\fillop] (\j,-\i) rectangle ++(1,-1);
  }

  \fill[pattern=north east lines,pattern color=black] (1, 0) rectangle ++(1,-1);
  \fill[pattern=north east lines,pattern color=black] (2, 0) rectangle ++(1,-1);
  \fill[pattern=north east lines,pattern color=black] (3, 0) rectangle ++(1,-1);
  \fill[pattern=north east lines,pattern color=black] (4, 0) rectangle ++(1,-1);
  \fill[pattern=north east lines,pattern color=black] (5, 0) rectangle ++(1,-1);
  \fill[pattern=north east lines,pattern color=black] (9, 0) rectangle ++(1,-1);
  \fill[pattern=north east lines,pattern color=black] (13,0) rectangle ++(1,-1);

  \fill[pattern=north east lines,pattern color=black] (1,-1) rectangle ++(1,-1);
  \fill[pattern=north east lines,pattern color=black] (2,-1) rectangle ++(1,-1);
  \fill[pattern=north east lines,pattern color=black] (3,-1) rectangle ++(1,-1);
  \fill[pattern=north east lines,pattern color=black] (4,-1) rectangle ++(1,-1);

  \fill[pattern=north east lines,pattern color=black] (1,-4) rectangle ++(1,-1);
  \fill[pattern=north east lines,pattern color=black] (5,-4) rectangle ++(1,-1);
  \fill[pattern=north east lines,pattern color=black] (9,-4) rectangle ++(1,-1);
  \fill[pattern=north east lines,pattern color=black] (13,-4) rectangle ++(1,-1);

  \fill[pattern=north east lines,pattern color=black] (1,-5) rectangle ++(1,-1);
  \fill[pattern=north east lines,pattern color=black] (5,-5) rectangle ++(1,-1);
  \fill[pattern=north east lines,pattern color=black] (6,-5) rectangle ++(1,-1);
  \fill[pattern=north east lines,pattern color=black] (7,-5) rectangle ++(1,-1);
  \fill[pattern=north east lines,pattern color=black] (8,-5) rectangle ++(1,-1);
  \fill[pattern=north east lines,pattern color=black] (9,-5) rectangle ++(1,-1);
  \fill[pattern=north east lines,pattern color=black] (13,-5) rectangle ++(1,-1);

  \fill[pattern=north east lines,pattern color=black] (5,-6) rectangle ++(1,-1);
  \fill[pattern=north east lines,pattern color=black] (6,-6) rectangle ++(1,-1);
  \fill[pattern=north east lines,pattern color=black] (7,-6) rectangle ++(1,-1);
  \fill[pattern=north east lines,pattern color=black] (8,-6) rectangle ++(1,-1);

  \fill[pattern=north east lines,pattern color=black] (1,-9) rectangle ++(1,-1);
  \fill[pattern=north east lines,pattern color=black] (5,-9) rectangle ++(1,-1);
  \fill[pattern=north east lines,pattern color=black] (9,-9) rectangle ++(1,-1);
  \fill[pattern=north east lines,pattern color=black] (13,-9) rectangle ++(1,-1);

  \fill[pattern=north east lines,pattern color=black] (1,-10) rectangle ++(1,-1);
  \fill[pattern=north east lines,pattern color=black] (5,-10) rectangle ++(1,-1);
  \fill[pattern=north east lines,pattern color=black] (9,-10) rectangle ++(1,-1);
  \fill[pattern=north east lines,pattern color=black] (10,-10) rectangle ++(1,-1);
  \fill[pattern=north east lines,pattern color=black] (11,-10) rectangle ++(1,-1);
  \fill[pattern=north east lines,pattern color=black] (12,-10) rectangle ++(1,-1);
  \fill[pattern=north east lines,pattern color=black] (13,-10) rectangle ++(1,-1);

  \fill[pattern=north east lines,pattern color=black] (9,-11) rectangle ++(1,-1);
  \fill[pattern=north east lines,pattern color=black] (10,-11) rectangle ++(1,-1);
  \fill[pattern=north east lines,pattern color=black] (11,-11) rectangle ++(1,-1);
  \fill[pattern=north east lines,pattern color=black] (12,-11) rectangle ++(1,-1);

  \fill[pattern=north east lines,pattern color=black] (1,-14) rectangle ++(1,-1);
  \fill[pattern=north east lines,pattern color=black] (5,-14) rectangle ++(1,-1);
  \fill[pattern=north east lines,pattern color=black] (9,-14) rectangle ++(1,-1);
  \fill[pattern=north east lines,pattern color=black] (13,-14) rectangle ++(1,-1);

  \foreach \i in {0,...,15}{
    \foreach \j in {0,...,15}{
      \node[font=\small] at (\j+0.5,-\i-0.5) {\EntryAt{\i}{\j}};
    }
  }

  \LatinGrid{16}
\end{tikzpicture}
}
}
\end{equation*}
\vspace*{0.5cm}

Next, for each order $n\in \{20,24,28\}$ we exhibit two transversals
$T^{(1)}_n$  and $T^{(2)}_n$ where $T^{(1)}_n\cap T^{(2)}_n=\emptyset$.

{\small
\begin{align*}
T^{(1)}_{20}=\{& (0,9,13), (1,1,5), (2,0,2), (3,3,6), (4,12,16), (5,14,19), (6,6,11), (7,11,18), (8,7,15), \\
&(9,8,17), (10,13,7), (11,19,10), (12,16,8), (13,10,3), (14,15,9), (15,17,12), (16,18,14), \\
& (17,4,1), (18,2,0), (19,5,4) \}, \\[1ex]
T^{(2)}_{20}=\{& (0,19,19), (1,2,2), (2,6,8), (3,13,16), (4,5,5), (5,12,17), (6,9,15), (7,7,14), (8,18,6),\\
& (9,3,12), (10,14,4), (11,11,1), (12,8,0), (13,16,9), (14,1,11), (15,15,10), (16,17,13), \\
&(17,10,7), (18,0,18), (19,4,3) \}, \\[1ex]
T^{(1)}_{24} = \{& (0, 3, 4), (1, 18, 19), (2, 8, 10), (3, 5, 8), (4, 9, 9), (5, 19, 0), (6, 7, 12), (7, 15, 22), (8, 17, 0),\\
& (9, 13, 18), (10, 4, 14), (11, 16, 3), (12, 1, 5), (13, 10, 22), (14, 21, 11), (15, 2, 17),  \\
&(16, 14, 19),(17, 23, 15), (18, 11, 17), (19, 20, 13), (20, 6, 2), (21, 0, 21), (22, 22, 18),\\
& (23, 12, 0) \}, \\[1ex]
T^{(2)}_{24} = \{ &(0, 9, 13), (1, 11, 12), (2, 12, 13), (3, 8, 11), (4, 0, 4), (5, 21, 22), (6, 17, 0), (7, 22, 5), \\
& (8, 19, 19), (9, 7, 16), (10, 5, 19), (11, 14, 0), (12, 6, 12), (13, 4, 14), (14, 20, 10), (15, 16, 8),\\
& (16, 10, 16),  (17, 3, 20), (18, 15, 9), (19, 13, 18), (20, 2, 22), (21, 18, 9), (22, 23, 20),\\
& (23, 1, 23)  \},  \\[1ex]
T^{(1)}_{28}=\{& (0,21,25), (1,3,3), (2,0,2), (3,14,17), (4,20,24), (5,13,22), (6,5,14), (7,11,18) ,\\
& (8,12,20), (9,6,15), (10,25,11), (11,1,12), (12,15,27), (13,22,7), (14,2,16), (15,19,6), \\
& (16,17,5), (17,9,26), (18,23,13), (19,18,9), (20,27,19), (21,7,0), (22,16,10),\\
& (23,26,21),  (24,8,4),  (25,4,1), (26,10,8),(27,24,23) \},  \\[1ex]
T^{(2)}_{28}=\{& (0,24,24), (1,27,0), (2,21,23), (3,4,7), (4,5,5), (5,3,8), (6,6,11), (7,23,2), (8,26,6), \\
&(9,13,18), (10,22,4), (11,11,21), (12,10,22), (13,7,20), (14,17,27), (15,2,17),  \\
& (16,9,25), (17,25,14), (18,1,19), (19,12,3), (20,20,12), (21,8,1), (22,15,9),  \\
&(23,18,13), (24,14,10), (25,19,16), (26,0,26), (27,16,15) \}.
\end{align*}
}

\medskip

Similarly to how we handled $\GG_n$ above, we used \aref{a:main} to
show that $\HH_n$ satisfies \tref{t:surrogate} for all $n=4k$ with
$9\le k\le2500$. For each of the entries in \Cref{Delta3} we found a
transversal that excludes the entry in question, but includes the
other two entries in \Cref{Delta3}.  In each case, we were able to
ensure there was no entry common to all three of these
transversals. Consequently we did not need to find a fourth
transversal, making this case slightly simpler than the $n\=2\bmod4$
case.

\medskip

Combining \lref{twomodfour} and \lref{zeromodfour},
we have thus established \tref{t:surrogate} for all orders except
$n \in \{28,32,34\}$. These exceptional orders will be handled in the
next section.

\section{Examples of Order 28, 32 and 34}\label{s:mopup}

Suppose that $n = 4k$, where $k\ge7$.
Consider the latin square $\XX_n$, given by
\begin{equation}\label{Specialcase1}
  \XX_n[a,b] = \begin{cases}
    a+b+ 3& \text{if $a\in \{0,4\}$ and $b-a \in \{1,5,8\},$}\\
    a+b+ 2 &  \text{if $a\in\{0,1\}$ and $10<b\equiv 1$ mod 2,} \\ 
    a+b+ 2 &  \text{if $a\in \{4,5\}$ and $b\equiv 1$ mod 2 and $b\notin\{5,7,9,11,13\},$} \\ 
    a+b+ 2 &  \text{if $a\in \{1,5\}$ and $b-a \in \{3,6,8\}$}, \\ 
    a+b+ 2& \text{if $a=8$ and $13\ne b\equiv 1$ mod 2},\\
    a+b+ 2&  \text{if $11 \le a < 3k-10 , a \equiv 2 $ mod 3 and $b \equiv 1$ mod 2,}\\
     a+b+1& \text{if $a\in \{1,5\}$ and $b-a \in \{0,1,2,5\},$}\\    
    a+b+ 1& \text{if $(a,b) \in \{(0,4),(2,3), (4,8),(6,7),(8,13),(8,14),(9,13)\}$},\\
    a+b - 1& \text{if $a\in \{2,6\}$ and $b-a \in \{0,4,7\},$}\\
    a+b - 1& \text{if $a\in \{3,7\}$ and $b-a \in \{1,2,5,6\},$ or $(a,b)=(9,14),$}\\
    a+b- 2&  \text{if $11 \le a < 3k-10 , a \equiv 1 $ mod 3 and $b \equiv 1$ mod 2,}\\ 
    a+b-2 &
\begin{aligned}[t]
&\text{if $a\in \{2,3\}$ and $9\ne b\equiv 1$ mod 2 and}\\
&\text{$(a,b)\notin\{(2,3),(2,7),(3,5)\}$,}
\end{aligned}\\
    a+b-2 &
\begin{aligned}[t]
&\text{if $a\in \{6,7\}$ and $13\ne b\equiv 1$ mod 2 and}\\
&\text{$(a,b)\notin\{(6,7),(6,11),(7,9)\}$,}
\end{aligned}\\
    a+b - 2& \text{if $a=10$ and $b\equiv 1$ mod 2,}\\
    a+b - 2& \text{if $a\in \{2,6\}$ and $b-a \in \{2,6\},$}\\
    a+b & \text{otherwise.}
  \end{cases}
\end{equation}

\begin{lemma}\label{orders2832}
  Let $n=4k$ for $k\ge7$.  The latin square $\XX_n$ does not have
  disjoint transversals.
\end{lemma}

\begin{proof}
  From \Cref{Specialcase1} we see that $-2 \le\Delta(r,c,s)\le 3$ for all entries
  $(r,c,s)$ in $\XX_n$.
The sum of the minimum $\Delta$-value from each row is
$$-1+\sum_{i=1}^{k-2} (-2)= -\dfrac{n}{2}+3 > - \dfrac{n}{2}.$$
Also, the sum of the maximum $\Delta$-value from each row is
$$2 \times (3+1)+1+\sum_{i=1}^{k-4}2=\dfrac{n}{2}+1,$$ which means
that every transversal of $\XX_n$ has to include an entry with the
maximum $\Delta$-value in all but one row. Based on
\Cref{Specialcase1}, the entries $$(2,3,6),\ (6,7,14),\ (9,13,23),$$
are the only ones with the maximum $\Delta$-value in rows $2$, $6$,
and $9$, respectively. This implies that every transversal must
include at least two of these three entries, so no pair of
transversals can be disjoint.
\end{proof}

Next, we present three transversals in each of $\XX_{28}$ and
$\XX_{32}$. In each case, there is no single entry that appears in all
three transversals.  We begin with three transversals in $\XX_{28}$.

{\small
\begin{align*}
\{&(0, 8, 11),
(1, 21, 24),(2, 14, 16),(3, 26, 1),(4, 5, 12),(5, 11, 18),(6, 7, 14),(7, 6, 13),(8, 9, 19),\\
&(9, 13, 23),(10, 20, 2),(11, 24, 7),(12, 15, 27),(13, 12, 25),(14, 1, 15),(15, 23, 10),(16, 17, 5),\\
&(17, 19, 8),(18, 16, 6),(19, 18, 9),(20, 25, 17),(21, 0, 21),(22, 10, 4),(23, 27, 22),(24, 2, 26),\\
&(25, 3, 0), (26, 22, 20),(27, 4, 3)
\},\\[1ex]
\{&
(0, 1, 4),(1, 11, 14),(2, 3, 6),(3, 12, 15),(4, 9, 16),(5, 23, 2),(6, 4, 10),(7, 6, 13),\\
&(8, 19, 1),(9, 13, 23),(10, 14, 24),(11, 17, 0),(12, 15, 27),(13, 24, 9),(14, 7, 21),(15, 10, 25),\\
&(16, 20, 8),(17, 16, 5),(18, 0, 18),(19, 21, 12),(20, 25, 17),(21, 18, 11),(22, 26, 20),(23, 27, 22),\\
&(24, 2, 26),(25, 22, 19),(26, 5, 3),(27, 8, 7)
\},\\[1ex]
\{&
(0, 1, 4),(1, 19, 22),(2, 3, 6),(3, 22, 25),(4, 12, 19),(5, 17, 24),(6, 7, 14),(7, 0, 7),\\
&(8, 5, 15),(9, 24, 5),(10, 8, 18),(11, 2, 13),(12, 11, 23),(13, 16, 1),(14, 26, 12),(15, 15, 2),\\
&(16, 10, 26),(17, 27, 16),(18, 21, 11),(19, 9, 0),(20, 18, 10),(21, 6, 27),(22, 14, 8),(23, 25, 20),\\
&(24, 13, 9),(25, 20, 17),(26, 23, 21),(27, 4, 3)
\}.
\end{align*}
}

\medskip

Similarly, we now present three transversals in $\XX_{32}$.

\medskip

{\small
\begin{align*}
\{&
(0,8,11), (1,17,20), (2,30,0), (3,12,15), (4,5,12), (5,1,8), (6,7,14), (7,2,9), (8,9,19),\\
& (9,13,23), (10,28,6), (11,21,2), (12,19,31), (13,20,1), (14,15,29), (15,3,18), (16,6,22),\\
& (17,11,28), (18,24,10), (19,18,5), (20,25,13), (21,14,3), (22,4,26), (23,16,7), (24,29,21), \\
&(25,23,16), (26,10,4), (27,22,17), (28,31,27), (29,27,24), (30,0,30), (31,26,25)
\},\\[1ex]
\{&
(0,5,8), (1,21,24), (2,3,6), (3,28,31), (4,9,16), (5,11,18), (6,14,20), (7,30,5), (8,29,7),\\
& (9,13,23), (10,24,2), (11,15,28), (12,7,19), (13,0,13), (14,12,26), (15,10,25), (16,16,0),  \\
&(17,26,11),(18,18,4), (19,25,12), (20,22,10), (21,6,27), (22,27,17), (23,31,22), (24,17,9), \\
&(25,4,29), (26,20,14), (27,8,3), (28,19,15), (29,1,30), (30,23,21), (31,2,1)
\},\\[1ex]
\{&
(0,1,4), (1,25,28), (2,3,6), (3,30,1), (4,5,12), (5,17,24), (6,7,14), (7,14,21), (8,19,29),\\
& (9,16,25), (10,6,16), (11,29,10), (12,18,30), (13,4,17), (14,26,8), (15,11,26), (16,31,15), \\
&(17,22,7), (18,23,9), (19,13,0), (20,0,20), (21,10,31), (22,12,2), (23,20,11), (24,27,19), \\
&(25,2,27), (26,28,22),(27,8,3), (28,9,5), (29,21,18), (30,15,13), (31,24,23)
\}.
\end{align*}
}

\medskip

We now present a latin square with the desired properties for order 34.

Consider the following latin square of order $34$:
\begin{equation}\label{Specialcase2}
  \YY_{34}[a,b] = \begin{cases}
    a+b+3& \text{if $a\in \{0,4,8\}$ and $b-a \in \{4,7,10\}$},\\
    a+b+2 &  \text{if $a\in \{0,4,8\}$ and $b\equiv 1$ mod 2 and $b-a\notin \{1,\dots,11\}$,} \\ 
    a+b+2 &  \text{if $(a,b)\in \{(0,2),(4,6),(8,10),(1,0),(5,4),(9,8)\}$,} \\ 
    a+b+1 &  \text{if $a\in \{1,2,5,6,9,10\}$ and $b\equiv 0$ mod 2 and $b-a\notin\{-2,\dots,9\}$}, \\ 
    a+b+1 &  \text{if $a\in \{0,2,4,6,8,10\}$ and $b-a=(-1)^{a/2}$}, \\ 
    a+b-1 &  \text{if $a\in \{1,2,5,6,9,10\}$ and $b\equiv 1$ mod 2 and $b-a\notin\{-1,\dots,10\}$}, \\
    a+b-1 &  \text{if $a\in \{2,6,10\}$ and $b-a=0$}, \\     
   a+b-1& \text{if $a\in \{1,5,9\}$ and $b-a\in \{0,1\}$},\\
    a+b-1& \text{if $a\in \{3,7,11\}$ and $b-a=-2$},\\
    a+b-2 &  \text{if $a\in\{3,7,11\}$ and $b\equiv 0$ mod 2 and $b-a\notin\{-1,\dots7\}$}, \\
    a+b-3& \text{if $a\in \{3,7,11\} $ and $b - a \in \{1,4,7\}$},\\
    a+b & \text{otherwise.}
  \end{cases}
\end{equation}

\begin{lemma}\label{orders34}
  The latin square $\YY_{34}$ has transversals,
  and any pair of transversals of $\YY_{34}$ have at least one entry in
  common, yet no entry is common to all transversals of $\YY_{34}$.
\end{lemma}

\begin{proof}
  From \Cref{Specialcase2} we see that $-3 \le\Delta(r,c,s)\le 3$ for
  all entries $(r,c,s)$ in $\YY_{34}$.
The sum of the minimum $\Delta$-value from each row is
$$3 \times (-1-1-3) = - 15 > -17 = - \dfrac{n}{2}.$$
The sum of the maximum $\Delta$-value from each row is
$$3 \times (3+2+1) = 18 = \dfrac{n}{2}+1,$$ which means that each
transversal has to include an entry with the maximum $\Delta$-value in
all but one row. Based on \Cref{Specialcase2}, the
entries $$(1,0,3),\ (5,4,11),\ (9,8,19),$$ are the only entries with
the maximum $\Delta$-value in rows $1$, $5$, and $9$,
respectively. This implies that every transversal must include at
least two of these entries to satisfy \lref{l:Delta}, and hence no
pair of transversals can be disjoint.

We now present three transversals in $\YY_{34}$.

{\small
\begin{align*}
\{&
(0,10,13), (1,20,22), (2,1,4), (3,27,30), (4,14,21), (5,4,11), (6,16,23), (7,29,2), (8,15,26), \\
& (9,8,19), (10,26,3), (11,33,10), (12,3,15), (13,28,7), (14,0,14), (15,13,28), (16,17,33), \\
& (17,12,29), (18,21,5), (19,24,9), (20,11,31), (21,30,17), (22,2,24), (23,23,12), (24,18,8), \\
&  (25,25,16),  (26,6,32), (27,7,0), (28,31,25), (29,32,27), (30,5,1), (31,9,6),  (32,22,20),\\
& (33,19,18)
\},\\[1ex]
\end{align*}
}
{\small
\begin{align*}
\{&
(0,7,10), (1,0,3), (2,32,1), (3,6,9), (4,14,21), (5,28,0), (6,5,12), (7,15,22), (8,12,23),\\
& (9,8,19), (10,4,15), (11,27,4), (12,21,33), (13,18,31), (14,26,6), (15,13,28), (16,31,13), \\
& (17,1,18),(18,11,29), (19,22,7), (20,16,2), (21,33,20), (22,17,5), (23,9,32), (24,2,26), \\
& (25,20,11), (26,24,16), (27,3,30), (28,23,17), (29,19,14), (30,29,25), (31,30,27), (32,10,8), \\
&(33,25,24)
\},\\[1ex]
\{&
(0,10,13), (1,0,3), (2,16,19), (3,25,28), (4,8,15), (5,4,11), (6,2,9), (7,31,4), (8,18,29), \\
& (9,32,8), (10,30,7), (11,5,16), (12,14,26), (13,11,24), (14,7,21), (15,12,27), (16,19,1), \\
&(17,23,6), (18,21,5), (19,17,2), (20,24,10), (21,13,0), (22,3,25), (23,28,17), (24,9,33),  \\
&(25,6,31),(26,22,14), (27,29,22), (28,26,20), (29,1,30), (30,27,23), (31,15,12), (32,20,18), \\
&(33,33,32) 
\}.
\end{align*}
}
By inspection, there is no entry common to all three of the above transversals.
\end{proof}

\section{Dominant transversals}\label{s:dom}

In this section we focus on the special situation where there is a
dominant transversal, namely a transversal that intersects all other
transversals in a given latin square.  As discussed in the
introduction, random latin squares typically do not have dominant
transversals.  We start by showing that two other highly structured
classes of latin squares never have dominant transversals.

\begin{theorem}\label{t:GroupTables}
If $L$ is the Cayley table of a finite group of order $n > 1$, then
$L$ does not have a dominant transversal.
\end{theorem}

\begin{proof}
  If $L$ has no transversals at all, then it clearly has no
  dominant transversal. So, suppose $L$ admits a transversal
  consisting of entries $(i,\star,g_i)$ chosen from row $i$. Consider
  a fixed element $g\in L$. If instead we select the entry
  $(i,\star,g_ig)$ from each row $i$, then the resulting set of
  entries again forms a transversal. Furthermore, as $g$ ranges over
  all elements of $G$, the transversals obtained in this way are
  pairwise disjoint, so none of them is dominant.
\end{proof}

A \emph{diagonally cyclic latin square} of order $n$ is a latin square
$L$ indexed by $\Z_n$ that possesses
$\alpha_n=(0,1,\ldots,n-1)$ as an automorphism of order
$n$. Equivalently, $L$ satisfies
$$L[i+1,j+1] \equiv L[i,j]+1\bmod n,$$
for all $i,j\in\Z_n$. Diagonally cyclic latin squares have proved
useful for many problems, including several involving orthogonality
and transversals \cite{Wan04}. However, they will not be useful
for the present purposes:

\begin{theorem}\label{t:diagcyc}
 A diagonally cyclic latin square has no dominant transversal.
\end{theorem}

\begin{proof}
Let $L$ be a diagonally cyclic latin square of order $n$.  The action
of $\alpha_n$ partitions the $n^2$ entries of $L$ into $n$ disjoint
orbits, each of size $n$. By construction, each orbit contains exactly
one entry from each row, each column, and each symbol, and hence each
orbit forms a transversal. Denote these transversals by
$D=\{T_1,\ldots,T_n\}$. None of the $T_i$ can be a dominant
transversal, as transversals in $D$ are pairwise disjoint. Suppose for
contradiction that there exists a transversal $T\notin D$ that is
dominant. Thus, by definition, $T$ must intersect every transversal of
$L$, and in particular it must meet each member of $D$. This forces
$T$ to contain exactly one element from each $T_i$, where $1\le i \le n$.

However, $\alpha_n(T)$ is also a transversal, and we have
$T\cap\alpha_n(T)=\emptyset$. This contradicts the assumption that $T$
is dominant.
\end{proof}

The above results rule out dominant transversals in some
important cases, but we will nevertheless be able to build some latin
squares that have dominant transversals. We propose the following question.

\begin{ques}\label{q:dom}
For every integer $n\ge 7$, does there exist a latin square of order $n$
that admits a dominant transversal.
\end{ques}

We note that there are some small orders for which no latin square has
a dominant transversal.  \tref{t:GroupTables} implies that no latin
square of order $2$, $3$ or $4$ has a dominant transversal, since all
latin squares of these orders are equivalent (up to permutation of the
rows and columns, and relabelling of the symbols) to a Cayley table of
a group of the same order. Also, it was noted in \cite{EW12} that
no latin square of order 6 has a dominant transversal. For many larger
orders we can answer \qref{q:dom} in the affirmative.

\begin{theorem}\label{t:not3mod4}
For every integer $n>4$ with $n\not\equiv 3 \bmod 4$ and $n\ne6$,
there exists a latin square of order $n$ that admits a dominant
transversal.
\end{theorem}

\begin{proof}
From \cite{EW12} we know that there is a unique species of latin
square of order 8 with a dominant transversal.  Here is a
representative of that species, with one of its three dominant
transversals highlighted. It has 24 transversals in total.
\[
\begin{tikzpicture}[scale=0.7]

\foreach \i/\j in {0/7,1/0,2/4,3/3,4/6,5/2,6/5,7/1} {
  \fill[SkyBlue!70!white] (\j,-\i) rectangle (\j+1,-\i-1);
}
\LatinGrid{8}
\foreach \i/\j/\num in {
  0/0/3, 0/1/7, 0/2/6, 0/3/0, 0/4/4, 0/5/5, 0/6/2, 0/7/1,
  1/0/7, 1/1/2, 1/2/1, 1/3/4, 1/4/5, 1/5/3, 1/6/6, 1/7/0,
  2/0/6, 2/1/1, 2/2/4, 2/3/5, 2/4/2, 2/5/7, 2/6/0, 2/7/3,
  3/0/0, 3/1/5, 3/2/7, 3/3/6, 3/4/1, 3/5/2, 3/6/3, 3/7/4,
  4/0/4, 4/1/3, 4/2/2, 4/3/7, 4/4/0, 4/5/1, 4/6/5, 4/7/6,
  5/0/5, 5/1/4, 5/2/3, 5/3/1, 5/4/6, 5/5/0, 5/6/7, 5/7/2,
  6/0/2, 6/1/6, 6/2/0, 6/3/3, 6/4/7, 6/5/4, 6/6/1, 6/7/5,
  7/0/1, 7/1/0, 7/2/5, 7/3/2, 7/4/3, 7/5/6, 7/6/4, 7/7/7
} \node at (\j+0.5,-\i-0.5) {\num};


\end{tikzpicture}
\]
For even $n>8$ the latin squares in \Cref{t:pinned} admit dominant
transversals; in fact, every transversal in those latin squares is
dominant.

Henceforth, we can assume that $n\equiv1\bmod4$.  For order $n=5$, the
following latin square admits three distinct transversals, highlighted
in three different colours.
\begin{equation}
\begin{tikzpicture}[scale=0.7]
\newcommand{\shadecell}[3]{\fill[#3,opacity=0.35] (#2,-#1-1) rectangle ++(1,1);}
\definecolor{tOne}{RGB}{228,26,28}   
\definecolor{tTwo}{RGB}{55,126,184}  
\definecolor{tThree}{RGB}{77,175,74} 

\foreach \i/\j in {0/2, 1/0, 2/3, 3/1, 4/4} {\shadecell{\i}{\j}{tOne}}
\foreach \i/\j in {0/4, 1/0, 2/1, 3/3, 4/2} {\shadecell{\i}{\j}{tTwo}}
\foreach \i/\j in {0/3, 1/0, 2/2, 3/4, 4/1} {\shadecell{\i}{\j}{tThree}}

\foreach \i/\j/\num in {
    0/0/0, 0/1/1, 0/2/2, 0/3/3, 0/4/4,
    1/0/1, 1/1/0, 1/2/3, 1/3/4, 1/4/2,
    2/0/2, 2/1/3, 2/2/4, 2/3/0, 2/4/1,
    3/0/3, 3/1/4, 3/2/1, 3/3/2, 3/4/0,
    4/0/4, 4/1/2, 4/2/0, 4/3/1, 4/4/3}
{
    \node at (\j+0.5,-\i-0.5) {\num};
}

\LatinGrid{5}

\end{tikzpicture}
\end{equation}
The entry $(1,0,1)$ is pinned, so each transversal is dominant.

Now suppose that $n=4k+1>5$ and let $\Gamma=\{0,1,\dots,2k-1\}\subset\Z_n$.
Let $M$ be a latin square of order $n$, indexed by $\Z_n$,
containing a latin subsquare $S$ of order $2k$ indexed by $\Gamma$.
Such an $M$ is not
hard to build, and existence can be deduced from a classic Theorem of
Ryser \cite{Rys51}. Note that $M$ decomposes into 4 blocks
\begin{align}\label{e:MannsTheorem}
M=\begin{bmatrix}
S&X\\
Y&Z\\
 \end{bmatrix}.
\end{align}
As $2k>2$, by replacing $S$ if necessary, we may assume that $S$ has a
transversal $T_1$.  Every symbol in $\Z_n\setminus\Gamma$ occurs $2k$
times within the rectangle $X$, and hence is missing from exactly one
column of $X$. Moreover, each column of $X$ is missing exactly one
symbol from $\Z_n\setminus\Gamma$.  Analogous properties hold for the
rows of $Y$.  Let $T_2$ be the set of $2k+1$ entries in $Z$ that have
symbols in $\Z_n\setminus\Gamma$. It follows that $T=T_1\cup T_2$ is a
transversal of $M$. We claim that $T$ is dominant.

The fact that $T$ is dominant can be deduced from a theorem of
Mann \cite{M44}, of which we now give a simple proof using a
variant of \lref{l:Delta}. Let
\[
\theta(x)=\begin{cases}
0&\text{if }x\in\Gamma,\\
1&\text{if }x\in\Z_n\setminus\Gamma,\\
\end{cases}
\]
and define $\Delta_2:M\rightarrow\Z_2$ by
$\Delta_2(r,c,s)=\theta(s)-\theta(r)-\theta(c)$. Note that $\Delta_2$ has
value zero on every entry of $M$ except for those in $T_2$.
However, in any transversal $T'$ of $M$,
\begin{equation}\label{e:Deltavar}
\sum_{(r,c,s)\in T'}\Delta_2(r,c,s)
\=\sum_{s}\theta(s)-\sum_{r}\theta(r)-\sum_{c}\theta(c)
\=2k+1\=1\bmod2.
\end{equation}
It follows that $T'$ must include at least one entry from $T_2$, showing
that $T$ is dominant.
\end{proof}

From \cite{EW12} we know that there are exactly 19 species of order 7
and 36\,007 species of order 9 that admit dominant transversals.
Every one of these 36\,007 species of order 9 contains a subsquare of
order 4 that possesses a transversal. The argument in
\eref{e:Deltavar} demonstrates that such latin squares have a dominant
transversal.  The fact that no latin square of order $9$ has a
dominant transversal for any other reason might hint that there are
large orders (necessarily of order $3\bmod4$) for which no latin
square has a dominant transversal.

Among the 19 species of order 7 with dominant transversals,
the numbers of all transversals are
$$3,7,7,11,12,13,14,15,15,19,19,19,20,22,23,23,25,31,33.$$
For order $n=7$, we exhibit two distinct latin squares that
contain dominant transversals:
\begin{equation}\label{e:ord7}
\begin{tikzpicture}[scale=0.7]

\foreach \i/\j in {3/3,4/4,5/5,6/6} {
  \fill[SkyBlue!70!white] (\j,-\i) rectangle (\j+1,-\i-1);
}
\LatinGrid{7}
\foreach \i/\j/\num in {
  0/0/1, 0/1/0, 0/2/2, 0/3/6, 0/4/5, 0/5/4, 0/6/3,
  1/0/0, 1/1/2, 1/2/1, 1/3/4, 1/4/3, 1/5/6, 1/6/5,
  2/0/2, 2/1/1, 2/2/0, 2/3/5, 2/4/6, 2/5/3, 2/6/4,
  3/0/6, 3/1/4, 3/2/5, 3/3/3, 3/4/1, 3/5/2, 3/6/0,
  4/0/5, 4/1/3, 4/2/6, 4/3/1, 4/4/4, 4/5/0, 4/6/2,
  5/0/4, 5/1/6, 5/2/3, 5/3/2, 5/4/0, 5/5/5, 5/6/1,
  6/0/3, 6/1/5, 6/2/4, 6/3/0, 6/4/2, 6/5/1, 6/6/6
} \node at (\j+0.5,-\i-0.5) {\num};

\end{tikzpicture}
\qquad
\begin{tikzpicture}[scale=0.7]

  \foreach \i/\j in {0,0,1/1,2/2,3/3,4/4,5/5,6/6} {
    \fill[SkyBlue!70!white] (\j,-\i) rectangle (\j+1,-\i-1);
  }

\LatinGrid{7}
\foreach \i/\j/\num in {
    0/0/0, 0/1/1, 0/2/2, 0/3/3, 0/4/4, 0/5/5, 0/6/6,
    1/0/1, 1/1/2, 1/2/0, 1/3/5, 1/4/6, 1/5/4, 1/6/3,
    2/0/2, 2/1/0, 2/2/1, 2/3/6, 2/4/5, 2/5/3, 2/6/4,
    3/0/3, 3/1/6, 3/2/5, 3/3/4, 3/4/0, 3/5/1, 3/6/2,
    4/0/4, 4/1/5, 4/2/6, 4/3/0, 4/4/3, 4/5/2, 4/6/1,
    5/0/5, 5/1/3, 5/2/4, 5/3/2, 5/4/1, 5/5/6, 5/6/0,
    6/0/6, 6/1/4, 6/2/3, 6/3/1, 6/4/2, 6/5/0, 6/6/5
}
{
    \node at (\j+0.5,-\i-0.5) {\num};
}
\end{tikzpicture}
\end{equation}
The first latin square in \eref{e:ord7} possesses only three
transversals, each of which includes the four highlighted pinned
entries.  The second latin square in \eref{e:ord7} contains 31
transversals, including 9 that are dominant, one of which is
highlighted.  It is the only species that has more than one subsquare
of order 3 among the 19 species with dominant transversals.  Three
other species contain exactly one subsquare of order 3, while the
remaining fifteen contain none.  Subsquares of order $3=(n-1)/2$ are
of interest given how we handled the $n\=1\bmod4$ case in
\tref{t:not3mod4}. However, the latin square of order 7 with the most
subsquares of order 3 comes from the Steiner quasigroup, and
\tref{t:diagcyc} shows that it does not have a dominant transversal.

\section{Concluding remarks}

\tref{t:surrogate}, \Cref{twomodfour} and \Cref{zeromodfour} all
support \cjref{cj:main}. However, we have not been able to identify a
general argument guaranteeing the existence of the transversals
required to prove the conjecture. It is entirely possible that
\cjref{cj:main} is true for some values of $n<28$. It would be most
interesting to know if it can be true for any odd value of $n$.

Transversal properties of latin squares of odd order seem to differ in
important ways from the even order case, and \lref{l:Delta} gives some
hint as to why the two cases are different. The infamous conjecture
known as Ryser's conjecture asserts that all latin squares of odd
order have a transversal. A latin square is equivalent
to a proper colouring of the edges of the complete bipartite graph
$K_{n,n}$ using $n$ colours. A transversal is a rainbow perfect
matching in such a colouring, namely, a perfect matching in which all
edges have different colours. Ryser's conjecture asserts that there
should be a rainbow perfect matching whenever $n$ is odd. However, if
a latin square has a dominant transversal it means that we can remove
a perfect matching from $K_{n,n}$, then properly colour the remaining
edges with $n$ colours in such a way that there is no rainbow perfect
matching. We are thus getting close, in one sense, to counterexamples
to Ryser's conjecture when we prove results like \tref{t:dom}.
This adds another motivation for studying dominant transversals.

It would be nice to know whether \tref{t:dom} generalises to large orders
that are $3\bmod4$.  Order $9$ is the largest order for which complete
enumerations of dominant transversals have been performed. We noted in
\sref{s:dom} that for this order the only latin squares with dominant
transversals, have dominant transversal for a reason that does not
apply to orders that are $3\bmod4$. This fact, together with our
failure to find examples of orders 11 or 15, makes us uncertain what we
expect the answer to \qref{q:dom} will be.

\subsection*{Acknowledgements} This research was supported by
  Australian Research Council grant DP250101611.
  
  \let\oldthebibliography=\thebibliography
  \let\endoldthebibliography=\endthebibliography
  \renewenvironment{thebibliography}[1]{%
    \begin{oldthebibliography}{#1}%
      \setlength{\parskip}{0.2ex}%
      \setlength{\itemsep}{0.2ex}%
  }%
  {%
    \end{oldthebibliography}%
  }

\end{document}